\documentclass[12pt]{article}

\usepackage{amsmath,amsthm}
\usepackage{amsfonts,xcolor,graphicx,url}

\usepackage[margin=2.5cm]{geometry}

\usepackage{titlesec}
\titleformat{\section}
{\normalfont\bfseries}{\thesection}{1em}{}
\titleformat{\subsection}
{\normalfont}{\thesubsection}{1em}{}

\numberwithin{equation}{section}

\newcommand{\euc}[1]{\mathbb{E}^{#1}}
\newcommand{\ik}{\mathbb{K}}
\newcommand{\iH}{\mathbb{H}}
\newcommand{\hu}{\hat{u}}
\newcommand{\hv}{\hat{v}}
\newcommand{\he}{\hat{e}}

\renewcommand{\vec}{}

\theoremstyle{plain}
\newtheorem{theorem}{Theorem}
\newtheorem{lemma}{Lemma}

\newtheorem*{corollary}{Corollary}

\theoremstyle{definition}
\newtheorem{definition}{Definition}

\theoremstyle{remark}

\begin{document}

\title{
Piecewise flat approximations of local extrinsic curvature for non-Euclidean embeddings
}

\author{Rory Conboye}

\date{}
\maketitle

\vspace{-0.8cm}
\begin{center}
\emph{Dedicated to the memory of Niall {\'O} Murchadha}
\end{center}
\vspace{-0.2cm}

\let\thefootnote\relax\footnotetext{
\hspace{-0.75cm}
	Department of Mathematics and Statistics
	\hfill conboye@american.edu \\
	American University
	 \\
	4400 Massachusetts Avenue, NW
	 \\
	Washington, DC 20016, USA
}
\let\thefootnote\svthefootnote

\begin{abstract}
	Discrete forms of the mean and directed curvature are constructed on piecewise flat manifolds, providing local curvature approximations for smooth manifolds embedded in both Euclidean and non-Euclidean spaces.
	The resulting expressions take the particularly simple form of a weighted scalar sum of hinge angles, the angles between the normals of neighbouring piecewise flat segments, with the weights depending only on the intrinsic piecewise flat geometry and a choice of dual tessellation. 
	The constructions are based on a new piecewise flat analogue of the curvature integral along and tangent to a geodesic segment, with integrals of these analogues then taken over carefully defined regions to give spatial averages of the curvature. 
	Computations for surfaces in both Euclidean and non-Euclidean spaces indicate a clear convergence to the corresponding smooth curvature values as the piecewise flat mesh is refined, with the former comparing favourably with other discrete curvature approaches.
	
\


\end{abstract}

\section{Introduction}


Our current concept of curvature first appeared as far back as the 1300's, when Nicole Oresme defined the curvature of a circle as the inverse of its radius \cite{MedMyst}. This was later extended to general planar curves as the inverse radius of the osculating (kissing) circle, the circle that best-fits the curve at a point. With the advent of Calculus, a number of equivalent definitions were found, such as the change in normal or tangent vectors with respect to arc-length, or the change in arc-length along a one-parameter family of offset curves. When a curve is discretized using piecewise linear segments, discrete curvature measures naturally arise from these definitions by taking small finite changes instead of differential limits.


Generalizing to higher dimensional manifolds in Euclidean space, the curvature definitions follow through with little alteration, though a direction must now be specified, and piecewise linear curves can be generalized to simplicial piecewise flat manifolds, formed by joining flat simplices (triangles, tetrahedra, etc.) along their edges or faces, known as hinges.  Directly discretizing the smooth curvature definitions often leads to ambiguities here, so many approaches also use some form of spatial averaging. 
A particularly rich variety of such approaches has been developed in the field of Discrete Differential Geometry \cite{Petitj02,DDG,AdvancesDDG}, mostly with an interest in image processing, architecture, and 3D object modelling.


The curvature of \emph{non}-Euclidean embeddings is particularly important in General Relativity, not least in determining the evolution of the spatial geometry in numerical formulations \cite{ADM,MTW,Alc}. However, there are a number of challenges associated with discrete approximations for non-Euclidean embeddings. For a start, piecewise flat segments will generally not be embeddable in a non-Euclidean space. Instead, piecewise flat approximations must be made \emph{intrinsically}, with edge-lengths defined by the lengths of geodesic segments on the manifold. Some of the embedding information can then be recovered by extending the piecewise flat approximation to the embedding space around it. In particular, hinge angles, the angles between the normals of neighbouring simplices, can be determined in this way.


While a piecewise flat approach to General Relativity has existed since the 1960's, known as Regge Calculus \cite{Regge}, local curvature approaches have only involved single hinge angles \cite{Brewin,KLM89ec,TraceK}, which fail to approximate smooth extrinsic curvature at a local level, even for highly symmetric models. Unfortunately, the approaches from Discrete Differential Geometry cannot easily be adapted to non-Euclidean embeddings either, since they rely on a Euclidean embedding for the addition of vectors and tensors or the defining of normal vectors at vertices (see section \ref{sec:OtherApproaches} for more details).


Here, some of the spatial averaging concepts common in Discrete Differential Geometry are used to provide local approximations of smooth extrinsic curvature, but using only the hinge angles and the intrinsic geometry of a piecewise flat manifold, which can be applied to non-Euclidean embeddings. The foundation for this approach is a new piecewise flat analogue of a curvature path integral, found by expressing the finite change of a unit tangent vector in terms of the hinge angles. Integrals of this analogue are then taken over carefully defined regions, resulting in curvature averages as weighted scalar sums of the hinge angles. Similar approaches have already been used for both the scalar and Ricci curvature \cite{PLCurv}, also giving an effective piecewise flat Ricci flow \cite{PLRFGowdy}.


The new curvature constructions are computed for a series of piecewise flat approximations of two irregular surfaces in Euclidean three-space, and for two different grid types for a surface embedded in a non-Euclidean Gowdy space \cite{Gowdy}. The resulting piecewise flat curvatures give reasonable approximations for low resolutions, and indicate a clear convergence to the corresponding smooth curvature values as the resolutions are increased. The results for the Euclidean-embedded surfaces also closely match with results from approaches in Discrete Differential Geometry.


The rest of the paper starts with some background material in section \ref{sec:Background}, defining smooth curvature, piecewise flat manifolds, and hinge angles, and outlining other discrete curvature approaches from Regge Calculus and Discrete Differential Geometry. Section \ref{sec:GenApp} begins with a motivation and general procedure for the current approach, then extends curvature integrals to piecewise flat manifolds, and proves the piecewise flat analogue of the curvature path integral based on these. The new curvature constructions are developed in section \ref{sec:New}, with the choice of spatial regions motivated and resulting expressions proved, given the curvature integral extensions. Details and results of the computations are provided in section \ref{sec:Comp}.

\section{Background}
\label{sec:Background}

\subsection{Smooth differential curvature}
\label{sec:Smooth}


The curvature $\kappa$ of a curve in two-dimensions can be defined as the rate of change of the unit tangent vector $\hat{T}$, or the angular change $\psi$ of $\hat{T}$, with respect to the arc-length parameter $s$,
\begin{equation}
	\kappa 
	:= \left|\frac{d \hat{T}}{d s}\right|
	 = \left|\frac{d \psi}{d s}\right|
	.
	\label{eq:kappaE2}
\end{equation}
For an $n$-dimensional manifold $M^n$, embedded in some ambient space $N^{n + 1}$, this approach can be generalized by using the rate of change of a unit tangent vector $\hv$, in the direction of $\hv$ itself. Taking the normal component of the result ensures that changes of $\hv$ within $M^n$ do not contribute, giving a directional curvature 
\begin{equation}
	\kappa (\hv)
	= \left< \nabla_{\hv} \hv, \hat{n} \right> ,
	\label{eq:kappa}
\end{equation}
where $\hat{n}$ is the unit normal vector to $M^n$, $\nabla$ is the covariant derivative associated with the Levi-Civita connection on $N^{n+1}$, and $\left< \cdot , \cdot \right>$ is the inner product on the tangent space of $N^{n+1}$ at any given point. When the ambient space is Euclidean, $\kappa (\hv)$ is equivalent to the inverse radius of the circle that best fits $M^n$ in the direction of $\hv$ at a given point.


It is common to measure the complete embedding curvature by generalizing $\kappa(\hv)$ to include changes of a different unit tangent vector $\hu$, in the direction of $\hv$. Different fields of study have different names for this curvature, including the second fundamental form $\alpha$, extrinsic curvature tensor $K_{a b}$, and shape operator $Q$. There are slight technical differences between the three but they are closely related, with
\begin{equation}
	\alpha (\hu, \hv)
	= K_{a b} \, \hu^a \hv^b
	= \left< Q(\hv), \hu \right>
	= \left< \nabla_{\hv} \hu, \hat{n} \right> .
	\label{eq:curvatures}
\end{equation}
Here, it will be more convenient to consider only $\kappa (\hv)$, which can still give the complete embedding curvature at each point $x \in M^n$. This can be done by giving the values of $\kappa (\hv)$ in the $n$ principal curvature directions, or for any set of $n$ linearly independent vector fields tangent to $M^n$ in a neighbourhood of $x$, along with the $(n-1)/2$ vector fields given by the difference between each pair of these. The mixed-argument values of the second fundamental form can be given in terms of these by using its bilinear and symmetry properties,
\begin{eqnarray}
 &&\alpha (\hu - \hv, \hu - \hv)
 = \alpha (\hu, \hu)
 - 2 \alpha (\hu, \hv)
 + \alpha (\hv, \hv)
 \nonumber \\ \Leftrightarrow \qquad
 &&\alpha (\hu, \hv)
 = \frac{1}{2}\big(
 \kappa(\hu)
 + \kappa(\hv)
 - \kappa(\hu - \hv)
 \big) .
 \label{eq:MixedForm}
\end{eqnarray}


Many applications also depend on the average of the directional curvature at each point, known as the mean curvature. This can be found by adding the directional curvatures for an orthonormal set of vector fields $\{ \he_i \}$ tangent to $M^n$,
\begin{equation}
	H
	= \frac{1}{n} \sum_i^n \kappa (\he_i)
	.
	\label{Hsmooth}
\end{equation}
Note that in much of the physics literature the mean curvature refers to the trace of the extrinsic curvature tensor and is denoted $K$, with $K = n H$. Also, for surfaces in $\euc{3}$, the Laplace-Beltrami operator is twice the mean curvature vector $H \hat{n}$. Integrating the mean curvature over all of $M^n$ gives the \emph{total} curvature, $H_{total} = \int_{M^n} H d V$.

\subsection{Piecewise flat submanifolds}


Piecewise flat manifolds are formed by joining flat Euclidean segments together. The most simple of these segments are $n$-simplices (line segments, triangles, tetrahedra, etc.), since their shape is entirely fixed by the lengths of their edges. Here, all piecewise flat manifolds are assumed to be simplicial, with their topology determined by the simplicial graph, and their intrinsic geometry by the set of edge-lengths. This is a major advantage of piecewise flat manifolds, with the topology isolated from the geometry, and the intrinsic geometry specified without any need for coordinates.


The specific embedding of an orientable piecewise flat manifold in $\euc{n+1}$ can be determined by the relative orientations of neighbouring simplices, with pairs of adjacent $n$-simplices \emph{hinging} along the co-dimension-one simplex that they share. These $(n-1)$-simplices are known as hinges, with the angle between the normals on either side of a hinge $h$ 
known as a hinge angle and denoted $\phi_h$. 
A hinge angle can also be defined as the angle between vectors that are tangent to neighbouring $n$-simplices, and orthogonal to the hinge joining them, see figure \ref{fig:hinge}.
Here, a hinge angle is considered positive if the $n$-simplices are concave, and negative if convex, relative to a given orientation.

\begin{figure}[h]
	\centering
	\includegraphics[scale=1]{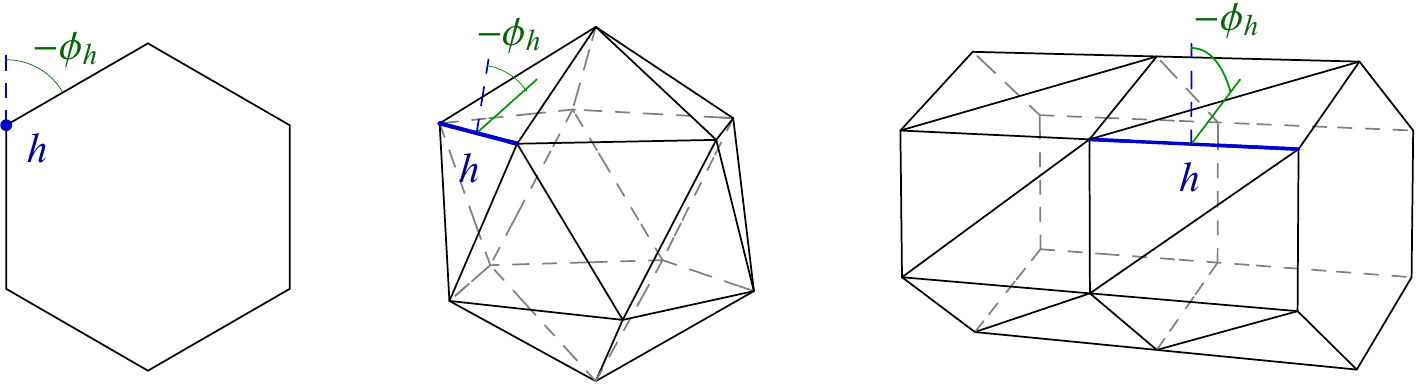}
	\caption{Triangulations and hinge angles for a circle in $\euc{2}$, and a sphere and cylinder in $\euc{3}$.}
	\label{fig:hinge}
\end{figure}


As with piecewise linear approximations of a curve, there are many ways to form a piecewise flat approximation, or triangulation, of a smooth manifold $M^n \subset \euc{n+1}$. The most common method is to take a set of points on $M^n$ and then create a piecewise flat manifold $S^n$ with these points as vertices, possibly with a global rescaling of $S^n$ to give an equivalent $n$-volume to $M^n$. In general, a good approximation will have uniformly small hinge angles.

\subsection{Non-Euclidean embeddings}
\label{sec:PFnonE}


Since flat $n$-simplices are generally not embeddable in a non-Euclidean space $N^{n + 1}$, a piecewise flat approximation of a smooth manifold $M^n \subset N^{n + 1}$ can only be made intrinsically. This is done by constructing a simplicial graph on $M^n$, using geodesic segments as edges, and defining a piecewise flat manifold $S^n$ to have the same graph, with its edge-lengths determined by the lengths of the corresponding geodesic segments in $M^n$.

\begin{figure}[h]
	\centering
	\includegraphics[scale=1]{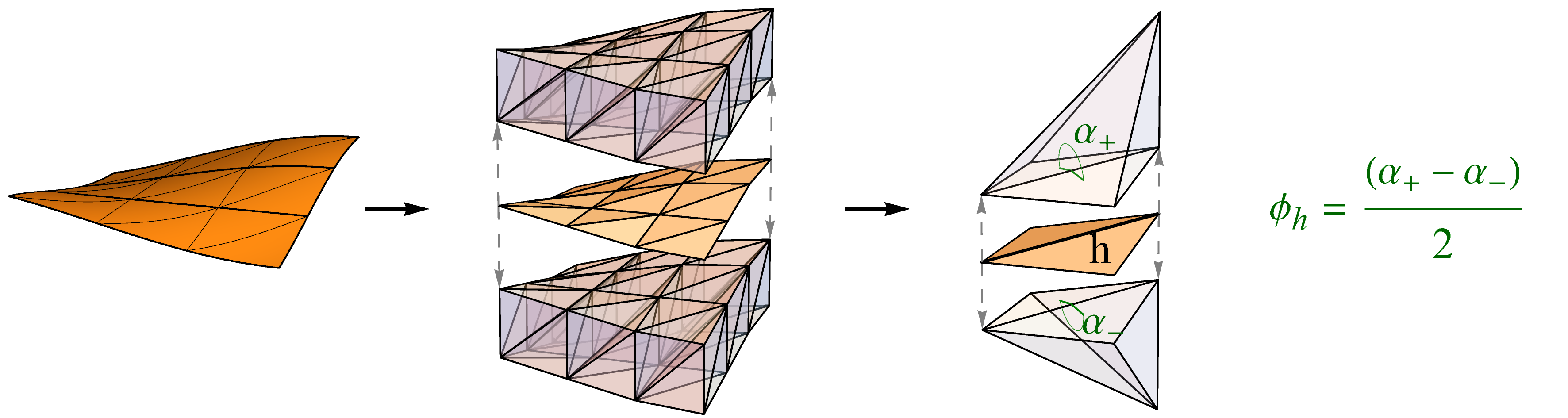}
	\caption{For a surface embedded in a non-Euclidean three-manifold, a pair of three-dimensional piecewise flat layers can be constructed on either side to specify the hinge angles.}
	\label{fig:NonEuclid}
\end{figure}


Information about the embedding of $M^n$ can be found by extending $S^n$ to form an $(n+1)$-dimensional piecewise flat approximation of $N^{n+1}$ in the region around $M^n$, see figure \ref{fig:NonEuclid}. The space around a hinge in $S^n$ will generally not be Euclidean, but it will be locally Euclidean on each side of $S^n$, consisting of a series of flat connected $(n+1)$-simplices. Hinge angles can then be determined in the usual way on each side, and an average of the two taken.

\subsection{Other discrete curvature approaches}
\label{sec:OtherApproaches}


The core of most piecewise flat curvature approaches is formed by a discrete analogue of a smooth curvature concept. Some examples include:
\begin{itemize}
	\item the hinge angles, the change in unit normal vector between neighbouring $n$-simplices;
	
	\item the change in vertex-based unit normal vectors along the length of each edge;
	
	\item the half-edge curvature, the inverse radius of a circular arc parallel to the midpoint and orthogonal to a normal vector at one end (see \cite{Petitj02}, Fig. 6 for a useful diagram).
\end{itemize}
These generally appear in weighted scalar, vector, or tensor sums, often resulting from a generalization of a smooth curvature measure. Some of the more prominent of these are briefly described below.

\emph{Steiner and Regge calculus:} Piecewise flat curvature analogues have existed since at least 1840, when Steiner \cite{Steiner} proposed an analogue of the total mean curvature for a convex polyhedron as the sum of its hinge angles, weighted by the corresponding edge lengths. 
This curvature analogue appeared as a coefficient in Steiner's formula for the volume enclosed by a surface offset from a convex polyhedron in $\euc{3}$. In the 1980's, Hartle and Sorkin \cite{HS81} gave a generalization of this formula, possibly unaware of Steiner's result, while introducing boundary terms to the Regge Calculus action. For the boundary of a piecewise flat four-manifold, the hinge angles were now weighted by the area of the hinge faces $|h_i|$,
\begin{equation}
	H_{total} = \sum_i |h_i| \phi_i .
	\label{HH81}
\end{equation}
Local curvatures have been suggested by taking a single hinge angle divided by a characteristic length \cite{KLM89ec}, or by dividing $|h|\phi_h$ by a hinge-based volume \cite{TraceK}. 
Unfortunately, these single-hinge approaches can fail to approximate smooth curvature values for even highly symmetric models, such as the cylinder triangulation in figure \ref{fig:hinge}, where the diagonal edges have a zero hinge angle, despite the corresponding smooth curvature being non-zero.

\emph{Cotan formula:} The cotan formula, first introduced by Pinkall and Polthier in \cite{PP93} is probably the most well developed method for locally approximating the mean curvature of smooth surfaces in $\mathbb{E}^3$. The formula is a weighted vector-sum of the edge-vectors around a vertex $v_i$, producing a mean curvature vector at the vertex. The contribution from each edge is equivalent to the inverse radius of the edge-based circle arc mentioned above, weighted by half the circumcentric area associated with the edge. The circumcentric areas can be expressed in terms of the cotangents of the two angles $\alpha$ and $\beta$ opposite the given edge, giving the formula
\begin{equation}
	H \hat{n}_i
	= \frac{1}{4 A_i}\sum_{v_j \in \star(v_i)} (\cot \alpha + \cot \beta) (\vec{v_i} - \vec{v}_j) ,
	\label{cotan}
\end{equation}
with $\vec{v_i}$ representing the vector in $\euc{3}$ for the given vertex, $\vec{v}_j$ the vectors for the vertices joined to it by an edge, and $A_i$ the Voronoi region dual to the vertex $v_i$. The formula can easily be extended to higher dimensions and co-dimensions \cite{MDSB03}, and some convergence properties have been studied in \cite{WarDDG,BS07}. 
While the interpretation given here relies on a triangulation being Delaunay, the approach appears to be more robust in practice, with the formula even defined for a mixed Voronoi-barycentric dual tessellation in \cite{MDSB03}.

\emph{Tensor sum of hinge angles:} In order to extend the Steiner formula to non-convex polyhedra, Cohen-Steiner and Morvan \cite{CSM03,CSM06} used the geometric measure theory concept of a normal cycle, a generalization of the normal bundle, essentially allowing for an overlapping of the offset for non-convex parts. Over a region $B$ of a piecewise flat submanifold of $\euc{3}$, the result is almost the same as a tensor sum of the hinge angles $\phi_i$, weighted by the length of the corresponding hinge lying within the region $B$. The expression provided in \cite{CSM03} for the anisotropic curvature measure over $B$ is
\begin{equation}
	\bar H (B) = \sum_{i} \frac{|h_i \cap B|}{2} \left[
	(\phi_i - \sin \phi_i) e_i^+ \otimes e_i^+
	+ (\phi_i + \sin \phi_i) e_i^- \otimes e_i^-
	\right],
	\label{CSM06}
\end{equation}
with $e_i^+$ and $e_i^-$ representing the normalized sum and difference of the unit normal vectors to the triangles joined along $h_i$. This expression is also shown to converge to the corresponding smooth curvature tensor under certain sampling conditions.

\emph{Other approaches:} Taubin \cite{Taub95} uses a weighted tensor sum of the same half-edge curvatures used to derive the cotan formula, but in contrast to Cohen-Steiner and Morvan, the sum is used to approximate an integral of the directed curvature over all directions at a point, rather than a spatial integral of the curvature tensor itself. The principal directions are still eignenvectors of the resulting tensor, but the principal curvatures need to be found from a pair of linear equations in the eigenvalues of the tensor. Instead of a tensor sum, Rusinkiewicz \cite{Russ04} uses a least-squares fit of the curvatures along the edges of a triangle, with the curvature measure coming from the change in vertex normal vectors along an edge.

\section{Preliminaries}
\label{sec:GenApp}

\subsection{General Approach}
\label{sec:GenApproach}

The aim of this paper is to use piecewise flat approximations of a smooth manifold to do two things: (i) to give effective local approximations of the extrinsic curvature; and (ii) to do so when the manifold is embedded in a non-Euclidean space. The first of these is achieved by a number of approaches in Discrete Differential Geometry, but all rely on a Euclidean embedding. This occurs either explicitly, such as adding vectors or tensors in $\euc{3}$, or implicitly, by requiring the construction of normal vectors at vertices, for example. In Regge Calculus, non-Euclidean embeddings are dealt with effectively for approximations of the total mean curvature, but attempts at localizing this can be seen to deviate from the smooth curvature even for highly symmetric models. 
While the approaches in both fields are unsuited to the aim of this paper, they have inspired a set of guiding principles outlined below.

\

\emph{Guiding principles:}
\begin{itemize}
	
	\item Integrals of smooth differential properties should be seen as entities in their own right, and it is these that should provide correlations between the piecewise flat and smooth. This follows an understanding from both Steiner and Regge, and has also been elucidated by Schr{\"o}der in \cite{SchroDDG}.	
	
	\item Averages over spatial $n$-volumes should be used to properly sample geometric properties on piecewise flat manifolds. This follows many of the Discrete Differential Geometry approaches, with an explicit argument for this given in \cite{MDSB03}.
	
	\item Geometric approximations should be associated with an appropriate type of graph element rather than setting up coordinate bases at vertices or within $n$-simplices. This embraces the coordinate-independent nature of piecewise flat manifolds that served as a motivation for Regge's original work \cite{Regge}.
	
\end{itemize}

Following the first principle, the smooth curvature integral along and tangent to a geodesic segment can be extended to piecewise flat manifolds by recognizing its equivalence to the total angular change of a tangent vector (shown in section \ref{sec:SmoothCurvIntegrals}), and re-defining the curvature integral accordingly. 
While this can be used to approximate the smooth curvature directly, an average can also be taken over a set of shorter geodesic segments, as suggested by the second principle, allowing for better approximations while intersecting fewer piecewise flat segments.
The orientation of these geodesic segments relative to a particular graph element can then be used to relate the curvature to a direction in the corresponding smooth manifold, fitting with the third principle. 
This leads to a general procedure for constructing piecewise flat curvature approximations.

\

\emph{General procedure:}

\begin{enumerate}
	\item Select a graph element to associate with a particular curvature.
	
	\item Define an $n$-volume region around this graph element, intercepting an appropriate sampling of hinges and capable of being intrinsically foliated into parallel line segments.
	
	\item Give the line integral of the curvature along each line segment, using the piecewise flat analogue of the cuvature integral along and tangent to a geodesic segment.
	
	\item Integrate these curvature line integrals over the foliation to give an $n$-volume curvature integral.

	\item Divide by the $n$-volume of the region to give an average of the curvature over this region.
	
\end{enumerate}

Before this procedure can be applied, in section \ref{sec:SmoothCurvIntegrals}, certain curvature integrals are re-defined so they can be used on piecewise flat manifolds, and in section \ref{sec:GeoTangent}, the piecewise flat analogue of the curvature integral along and tangent to a geodesic segment is given. The procedure above is then used to construct the piecewise flat mean and directed curvatures in section \ref{sec:New}.

\subsection{Extending curvature integrals to piecewise flat manifolds}
\label{sec:SmoothCurvIntegrals}

The smooth concept of curvature at a point clearly does not apply to piecewise flat manifolds, giving zero values within each $n$-simplex and infinite values at the hinges. However, lemma \ref{lem:SmoothGeoInt} below relates a \emph{path integral} of the curvature on a smooth manifold to the finite angular change in a tangent vector, a concept that \emph{is} well-defined on piecewise flat manifolds. In fact, the hinge angles themselves can be viewed as the finite change in a tangent vector along a path intercepting the corresponding hinge.

In definition \ref{def:geoInt}, this curvature path integral is then re-defined \emph{as} the finite angular change in the tangent vector, so that it can be applied to piecewise flat manifolds. This forms the bases for extensions to two spatial curvature integrals that follow.

\begin{lemma}[Smooth curvature integral along a geodesic segment]
	\label{lem:SmoothGeoInt}
	On a smooth manifold $M^n \subset N^{n + 1}$, the integral of the curvature along and tangent to a geodesic segment $\gamma$ is given by the total angular change in the unit tangent vector $\hat{v}$ to $\gamma$,
	\begin{equation}
		\int_\gamma \kappa(\hat{v}) ds
		= \int_\gamma d\psi
		,
	\end{equation}
	where $\psi$ measures the angular change in $\hat{v}$ along $\gamma$.
\end{lemma}

\begin{proof}
	From the definition of the directional curvature in (\ref{eq:kappa}),
	\begin{equation}
		\kappa(\hat{v})
		= \left<
		\nabla_{\hat{v}} \hat{v}, \, \hat n \right> .
	\end{equation}
	Since $\hat{v}$ is tangent to a geodesic in $M^n$, the geodesic equation states that $\nabla^{M}_{\hat v} \hat{v} = 0$, where $\nabla^{M}$ is the restriction of $\nabla$ to $M^n$. As a result, $\nabla_{\hat{v}} \hat{v}$ must be in a direction normal to $M^n$ at each point of $\gamma$. The covariant derivative $\nabla_{\hv} \hv$ is the rate of change of $\hv$ with respect to the arc-length parameter $s$ along $\gamma$, so
	\begin{equation}
		\nabla_{\hat{v}} \hat{v}
		= \frac{d \hat{v}}{d s}
		= \lim_{\Delta s \rightarrow 0} \frac{\Delta \psi \, \hat{n}}{\Delta s}
		= \frac{d \psi}{d s} \hat{n}
		,
	\end{equation}
	where the third term above assumes that $\Delta \psi$ is the angular difference between the tangent vector at a point, and the parallel transport of the tangent vector a distance of $\Delta s$ along $\gamma$. The resulting angular change in the tangent vector also works in the same way as a curve in $\euc{2}$, as shown in equation (\ref{eq:kappaE2}). The sign of the derivative is considered positive if the change is in the positive $\hat{n}$ direction, and negative otherwise. The integral of the curvature along $\gamma$ then becomes
	\begin{equation}
		\int_\gamma \kappa(\hat{v}) \, d s
		= \int_\gamma \left<
		\nabla_{\hat{v}} \hat{v}, \, \hat{n} \right> \, d s
		= \int_\gamma \left<
		\frac{d \psi}{d s} \, \hat{n}, \, \hat{n} \right> \, d s
		= \int_\gamma \frac{d \psi}{d s} \, d s
		= \int_\gamma d \psi
		,
	\end{equation}
	giving the total angular change in the tangent vector $\hat{v}$ along the geodesic segment $\gamma$.
\end{proof}

\begin{definition}[Geodesic-tangent curvature integral]
	\label{def:geoInt}
	For both a smooth manifold $M^n \subset N^{n + 1}$, and a piecewise flat approximation $S^n$ of $M^n$, the geodesic-tangent curvature integral associated with a geodesic segment $\gamma$ is defined as the total angular change in the tangent vector to $\gamma$ along its length, and denoted $\ik_\gamma$.
\end{definition}

\

If these are to be used to determine the average curvature value over an $n$-volume region, the corresponding $n$-volume integrals must be defined in terms of curvature path integrals instead of the curvature at a point.

\begin{lemma}[Smooth directed curvature integrals]
	\label{lem:SmoothDirectedInt}
	For an $n$-dimensional region $B$ of a smooth manifold $M^n \subset N^{n + 1}$, with a vector field $\hv$ and integral curves $c$ to $\hv$ that foliate $B$, the integral of the curvature in the direction of $\hv$ over $B$,
	\begin{equation*}
		\int_B \kappa(\hv) \ dV^n \
		= \ \int_{B/c} \left[ \int_c \kappa (\hv) ds \right] d V^{n - 1}.
	\end{equation*}
\end{lemma}

\begin{proof}
	Since $B$ is foliated by the curves $c$, the $n$-volume integral of $\kappa (\hv)$ over $B$ can be performed iteratively, with the integrals along the curves found first, and the results of these then integrated over the space $B/c$ orthogonal to $c$ in $B$.
\end{proof}

\begin{definition}[$n$-volume directed curvature integral]
	\label{def:DirectedInt}
	Take an $n$-dimensional region $B$ of either a smooth manifold $M^n \subset N^{n + 1}$, or a piecewise flat approximation $S^n$ of $M^n$, with a vector field $\hv$ and integral curves $c$ to $\hv$ that foliate $B$. The directed curvature integral over $B$, in the direction of $\hv$, is defined as the $(n - 1)$-dimensional integral of the curvature path integrals along and tangent to the curves $c$, denoted $\ik_c (\hv)$, where they are defined,
	\begin{equation*}
		\ik_B (\hv) 
		= \int_{B/c} \ik_c (\hv) d V^{n - 1},
	\end{equation*}
	with $B/c$ representing the space orthogonal to the curves $c$ in $B$.
\end{definition}

\

On a smooth manifold, an orthonormal basis of vector fields can be constructed on a neigbourhood of any differential manifold up to dimension three, \cite{DiagMetric}. On such a neighbourhood, an $n$-volume integral of the mean curvature can be given as a sum of directed curvature integrals, as shown in Lemma \ref{lem:SmoothMeanInt} below.

\begin{lemma}[Smooth mean curvature integrals]
	\label{lem:SmoothMeanInt}
	For a region $B \subset M^n \subset N^{n + 1}$, which can be decomposed into subregions $D_j$ on which orthonormal sets of basis vector fields $\{\he_i\}$ can be defined, the integral of the mean curvature over $B$,
	\begin{equation*}
		\int_B H \ dV \
		= \ \frac{1}{n} \sum_j \sum_i \int_{D_j} \kappa (\he_i) d V^n
		= \ \frac{1}{n} \sum_j  \sum_{i = 1}^n \ik_{D_j} (\he_i).
	\end{equation*}
\end{lemma}

\begin{proof}
	On each subregion $D_j$, since the vector fields $\he_i$ are all orthogonal to each other over $D_j$, the sum over the basis vector fields and the integral over $D_j$ commute, so
	\begin{equation*}
		\int_{D_j} H d V^n 
		= \int_{D_j} \frac{1}{n} \sum_{i = 1}^n \kappa (\he_i) d V^n
		= \frac{1}{n} \sum_{i = 1}^n \int_{D_j} \kappa (\he_i) d V^n
		= \frac{1}{n} \sum_{i = 1}^n \ik_{D_j} (\he_i).
	\end{equation*}
	with the final term coming from definition \ref{def:DirectedInt}, since $D_j$ is foliated by the integral curves of each coordinate vector field. As a scalar field, the mean curvature is invariant to the choice of basis vector fields, so the total mean curvature over $B$ is just the sum of the integrals for each subregion $D_j$.
\end{proof}

\begin{definition}[$n$-volume mean curvature integral]
	\label{def:MeanInt}
	For an $n$-dimensional region $B$ of either a smooth manifold $M^n \subset N^{n + 1}$, or a piecewise flat approximation $S^n$ of $M^n$, which can be decomposed into regions $D_j$ on which orthonormal sets of basis vector fields $\{\he_i\}$ can be defined, the integral of the mean curvature over $B$,
	\begin{equation}
		\iH_B
		= \ \frac{1}{n} \sum_j  \sum_{i = 1}^n \ik_{D_j} (\he_i).
	\end{equation}
\end{definition}

\

\subsection{Piecewise flat geodesic-tangent curvature integral}
\label{sec:GeoTangent}

While definition \ref{def:geoInt} above can be used to give the hinge angle $\phi_h$ on a piecewise flat manifold when the path $\gamma$ is orthogonal to the hinge $h$, it will prove more useful to find the curvature integral for geodesic paths that intercept hinges at different angles.

\begin{theorem}[Geodesic-tangent curvature integral on a piecewise flat manifold]
	\label{thm:geoIntS}
	For a path $\gamma$ that is intrinsically straight within a piecewise flat approximation $S^n$ of a smooth manifold $M^n \subset N^{n + 1}$, the geodesics-tangent curvature integral
	\begin{equation} 
		\ik_\gamma
		= \sum_h \cos \theta_h \ \phi_h + O(\phi_h^3) ,
	\end{equation}
	for all hinges $h$ that intersect $\gamma$, with $\phi_h$ giving the corresponding hinge angle, and $\theta_h$ the angle that $\gamma$ makes with the normal vector to the hinge $h$ within $S^n$.
\end{theorem}

\begin{proof}
	For a piecewise flat manifold $S^n \subset \euc{n + 1}$, the unit tangent vector $\hat v$ to $\gamma$ will remain constant within each $n$-simplex, so any changes will occur at the hinges. For each $n$-simplex $\sigma^n$ containing a hinge $h$ that intersects $\gamma$, a set of coordinate vectors $\{\hat w_i\}$ can be defined so that $\hat w_1$ is orthogonal to $h$ and tangent to $\sigma^n$, and $\hat w_2$ is aligned with the component of $\hat v$ tangent to $h$, so that
	\begin{equation}
		\hat v
		= \cos \theta_h \ \hat w_1
		+ \sin \theta_h \ \hat w_2.
	\end{equation}
	Taking two $n$-simplices $\sigma^n_A$ and $\sigma^n_B$ joined by the hinge $h$, their corresponding vectors $\hat w_1^A$ and $\hat w_1^B$ will differ by the hinge angle $\phi_h$.  
	To compare the tangent vector $\hat v$ on either side of $h$, a new set of coordinate vectors $\{\hat e_i\}$ can be defined, with 
	\begin{align}
		\hat e_1 
		&= \hat w_1^B - \hat w_1^A/
		\left|\hat w_1^B - \hat w_1^A\right|
		, \nonumber \\
		\hat e_2 
		&= \hat w_1^A + \hat w_1^B/
		\left|\hat w_1^A + \hat w_1^B\right|
		, \nonumber \\
		\hat e_3 
		&= \hat w_2	
		.
	\end{align}
	In this bases, the unit tangent vectors to $\gamma$ within $\sigma^n_A$ and $\sigma^n_B$ are
	\begin{align}
		\hat v_A 
		=
		- \sin \frac{\phi_h}{2} \, \cos \theta_h \ \hat e_1 +
		  \cos \frac{\phi_h}{2} \, \cos \theta_h \ \hat e_2 +
		  \sin \theta_h \ \hat e_3
		, \nonumber \\
		\hat v_B 
		= \ \,
		\ \sin \frac{\phi_h}{2} \, \cos \theta_h \ \hat e_1 +
		  \cos \frac{\phi_h}{2} \, \cos \theta_h \ \hat e_2 +
		  \sin \theta_h \ \hat e_3
		.
	\end{align}
	The angular change $\psi_h$ from $\hat v_A$ to $\hat v_B$ can now be found using the cross product of these,
	\begin{align}
		\sin \psi_h
		&= \left|\hat v_A \times \hat v_B \right|
		\nonumber \\
		&= \left|
		2 \sin \frac{\phi_h}{2} \, \cos \theta_h \, \sin \theta_h \ 
		\hat e_2 
		- 2 \sin \frac{\phi_h}{2} \, \cos \frac{\phi_h}{2} \, \cos^2 \theta_h \
		\hat e_3
		\right|
		\nonumber \\
		&= 2 \sin \frac{\phi_h}{2} \, \cos \theta_h \
		\sqrt{\sin^2 \theta_h
			+ \cos^2 \frac{\phi_h}{2} \, \cos^2 \theta_h}
		.
	\end{align}
	Solving for $\psi_h$, and Taylor expanding the result in terms of $\phi_h$ about zero,
	\begin{align}
		\psi_h
		&= \cos \theta_h \ \phi_h
		- \frac{1}{24} \, \cos \theta_h \, \sin^2 \theta_h \ \phi_h^3
		+ O(\phi_h^5).
		\nonumber \\
		&= \cos \theta_h \ \phi_h
		+ O(\phi_h^3).
		\label{eq:psih}
	\end{align}
	Figure \ref{fig:PFpath} also shows that $\sin \left(\frac{1}{2}\psi_h\right) = \cos \theta_h \,  \sin \left(\frac{1}{2} \phi_h\right)$, which requires nothing more than the small angle approximation for sine to give the leading term above.
	
	\
	
	\begin{figure}[h]
		\centering
		\includegraphics[scale=1]{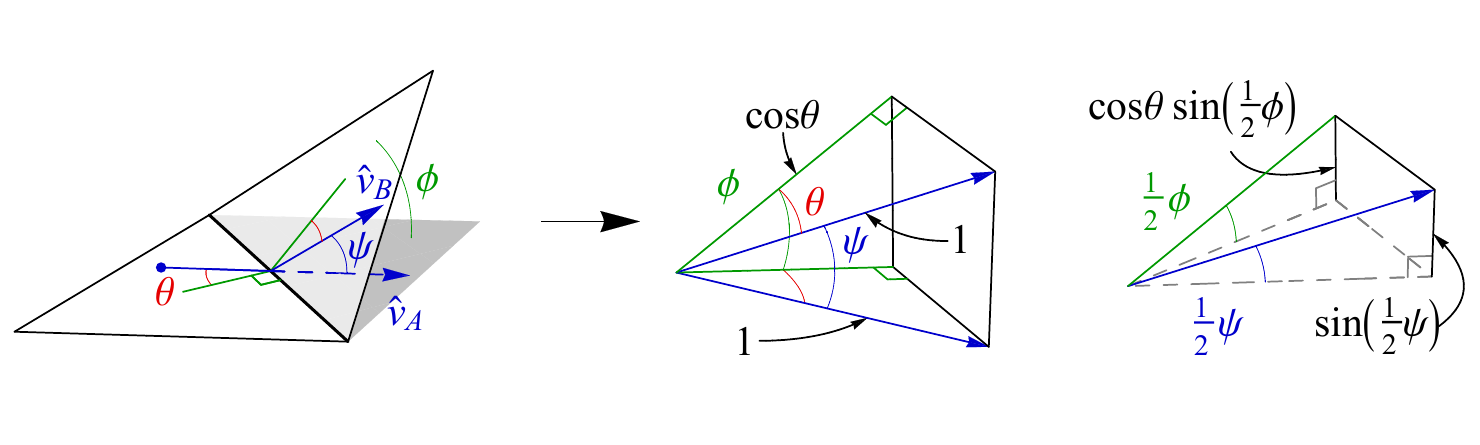}
		\caption{The two triangles on either side of a hinge are shown on the left, with a geodesic path $\gamma$ and the unit tangent vectors $\hv_A$ and $\hv_B$ to $\gamma$ for each triangle. The angle $\psi$ between the two tangent vectors can then be related to the hinge angle $\phi$ and the angle $\theta$.}
		\label{fig:PFpath}
	\end{figure}

	When $S^n$ is a piecewise flat approximation of a manifold $M^n \subset N^{n + 1}$, and the hinge angles are determined by extending $S^n$ to produce an $(n + 1)$-dimensional piecewise flat approximation for the region around $M^n$, as described in section \ref{sec:PFnonE}, the same procedure can be used for the flat spaces on either side of $S^n$. Averaging the resulting angles also gives (\ref{eq:psih}) in terms of the averaged hinge angle $\phi_h$.
	
	Summing the angular change $\psi_h$ for each hinge $h$ intersecting $\gamma$ gives the total angular change in $\hat v$ along the length of $\gamma$,
	\begin{equation}
		\ik_\gamma
		= \sum_h \psi_h
		= \sum_h \cos \theta_h \ \phi_h + O(\phi_h^3)
		,
	\end{equation}
	showing the geodesic-tangent curvature integral as a weighted sum of hinge angles.
\end{proof}

\section{New curvature constructions}
\label{sec:New}

\subsection{Piecewise flat mean curvature}
\label{sec:Hv}

The mean curvature has a single value at each point of a smooth manifold $M^n \subset N^{n + 1}$, and can be seen as an average of the directional curvature over \emph{all} directions on $M^n$. To construct an approximation of the smooth mean curvature on a peicewise flat manifold $S^n$, the $n$-volume regions should have the following characteristics:
\begin{itemize}
	\item the regions should tessellate $S^n$;
	
	\item each region should intercept hinges with a wide distribution of orientations;
	
	\item the resulting curvature constructions should change continuously for small variations in the triangulation $S^n$.
\end{itemize}
A tessellation ensures that each point in $S^n$ has a single mean curvature associated with it, a wide distribution of hinges can help to approximate an average curvature over all directions, and the final characteristic ensures that there are no jumps in a curvature value for small changes in the triangulation. 

The graph elements associated with the most general hinge orientations are the $n$-simplices and vertices. However, while the $n$-simplices provide a natural tessellation of $S^n$, they only contain hinges at their boundaries and will generally have a smaller sampling of hinges than the vertices. Thus, a dual tessellation, which decomposes $S^n$ into $n$-volume regions surrounding each vertex, will be used to construct approximations of the mean curvature.

\begin{definition}[Vertex regions]
	\label{def:V_v}
	A decomposition of a piecewise flat manifold $S^n$ into $n$-dimensional regions $V_v$ dual to each vertex $v$ is defined so that:
	\begin{enumerate}
		\item the vertex $v \in V_v$, but no other vertices are contained within $V_v$;
		
		\item the regions $V_v$ form a complete tessellation of $S^n$,
		\begin{equation}
			|S^n| = \sum_{v \in S^n} |V_v|, \qquad
			V_{v_i} \cap V_{v_j} = \emptyset \quad
			\forall \ i \neq j ,
		\end{equation}
		with $|S^n|$ and $|V_v|$ representing the $n$-volumes of $S^n$ and $V_v$ respectively;
		
		\item only hinges $h_v$ in the star of $v$ (containing $v$ in their boundary) intersect $V_v$.
	\end{enumerate}
\end{definition}

This gives the most general definition of a dual tessellation for the constructions that follow. The third property is not strictly necessary but is deemed a reasonable requirement, with Voronoi tessellations satisfying this condition only where $S^n$ gives a Delaunay triangulation for example. It also seems reasonable to require an additional property which distributes the total $n$-volume of $S^n$ in some consistent manner over the subregions $V_v$. However, there are many different methods for doing this, most of which are incompatible with one another, such as the Voronoi and barycentric tessellations. Such a property will therefore not be imposed here.

\begin{theorem}[Piecewise flat mean curvature]
	\label{thm:H_v}
	Assuming the curvature integrals in definitions \ref{def:geoInt}, \ref{def:DirectedInt} and \ref{def:MeanInt}, the average mean curvature over the region $V_v$ dual to a vertex $v$ on a piecewise flat manifold $S^n$, 
	\begin{equation}
		\label{H_v}
		H_v
		:= \frac{1}{|V_v|} \, \iH_{V_v}
		= \frac{1}{n \, |V_v|}
		\sum_{h \subset \mathrm{star}(v)} |h \cap V_v| \phi_h ,
	\end{equation}
	with $|h \cap V_v|$ representing the $(n - 1)$-volume of $h \cap V_v$, and $\mathrm{star}(v)$ representing all of the hinges in the star of $v$, or containing $v$ in their boundary.
\end{theorem}

\begin{proof}
	The vertex region $V_v$ can be decomposed into subregions $D_h$, each enclosing the intersection of a single hinge $h$ with $V_v$. Since each of these regions is composed of a pair of $n$-simplices, the subregions $D_h$ must be intrinsically flat within $S^n$. A Cartesian coordinate basis $\{ \he_i \}$ can therefore be defined on each subregion, with integral curves $\gamma_i$ given by line-segments within $S^n$. By definitions \ref{def:MeanInt} and \ref{def:DirectedInt}, the mean curvature integral over each subregion,
	\begin{equation}
		\mathbb{H}_{D_h}
		= \frac{1}{n} \sum_{i = 1}^n \ik_{D_h} (\he_i)
		= \frac{1}{n} \sum_{i = 1}^n \int_{{D_h}/\gamma_i} \ik_{\gamma_i} \ d V^{n - 1}
		.
	\end{equation}
	Choosing the coordinate vector $\he_1$ to be orthogonal to the hinge $h$, all of the line-segments except $\gamma_1$ will remain within a single $n$-simplex, and only the geodesic-tangent curvature integral along $\gamma_1$ will be non-zero, so by theorem \ref{thm:geoIntS},
	\begin{equation}
		\mathbb{H}_{D_h}
		= \frac{1}{n} \ \int_{{D_h}/\gamma_1} \ik_{\gamma_1} \ d V^{n - 1}
		= \frac{1}{n} \ \int_{{D_h}/\gamma_1} \phi_h \ d V^{n - 1}
		= \frac{1}{n} \ |h \cap V_v| \phi_h ,
	\end{equation}
	with the final integral resulting in the $(n-1)$-volume of the intersection of $h$ with $D_h$, since this is the space orthogonal to the line-segments $\gamma_1$ within $D_h$, and the hinge angle $\phi_h$ is invariant over $h$. The total integrated curvature over $V_v$ is given by summing over all of the subregions $D_h$,
	\begin{equation}
		\mathbb{H}_{V_v}
		= \sum_{h \subset \mathrm{star}(v)}
		\mathbb{H}_{D_h}
		= \sum_{h \subset \mathrm{star}(v)} \frac{1}{n} |h \cap V_v| \phi_h ,
	\end{equation}
	and the average mean curvature found by dividing this by the volume of $V_v$.
\end{proof}

For continuous variations in the triangulation $S^n$ of a smooth manifold $M^n$, where the simplicial graph remains fixed and the triangulation well-defined, the hinge angles and edge lengths will change continuously. For a barycentric dual tessellation, the $n$-volumes $|V_v|$ will also change continuously, and therefore so will the mean curvature $H_v$. For a Voronoi dual, the triangulation must also remain Delaunay. 

The construction also gives the same total mean curvature expression as that of Steiner \cite{Steiner} and Hartle and Sorkin \cite{HS81}, as shown below.

\begin{corollary}[Total mean curvature]
	\label{cor:IntH}
	The total mean curvature over a piecewise flat manifold $S^n$ is
	\begin{equation}
		\iH_{S^n}
		= \frac{1}{n} \sum_{h \subset S^n} |h| \phi_h .
		\label{TotalH}
	\end{equation}
\end{corollary}

\begin{proof}
	The total mean curvature over $S^n$ is equal to the sum of the mean curvature integrals for each vertex region $V_v$. Since these regions form a tessellation of $S^n$, each part of a given hinge $h$ must be contained in a single region. With the hinge angles $\phi_h$ fixed over $h$, the overall contribution from each hinge is $\frac{1}{n} |h| \phi_h$. The total integral is then given by the sum of these terms for all hinges $h \subset S^n$.
\end{proof}

\subsection{Piecewise flat directed curvature}
\label{sec:Kh}

There are two types of piecewise flat graph elements that have a direction associated with them, the edges and the $(n - 1)$-simplices, or hinges. For a curvature derived from the hinge angles, the hinges provide a natural choice for the construction of a piecewise flat directed curvature. An $n$-volume region defined around a hinge should then have the following characteristics:
\begin{itemize}
	\item the regions should intersect more than one hinge;
	
	\item they should not tessellate $S^n$;
	
	\item they should be intrinsically flat;
	
	\item the resulting curvature constructions should change continuously for small variations in the triangulation $S^n$.
\end{itemize}
A sampling of more than one hinge is required, as piecewise flat approximations can often lead to flat quadrilaterals, where a zero hinge on the diagonal results more from the graph structure than the curvature of the smooth manifold. See the regular triangulation of the cylinder in figure \ref{fig:hinge}, for example. The hinge regions should not form a tessellation of $S^n$, as each point on a smooth manifold has $n$ principle curvatures, so points on $S^n$ should also be contained in more than one hinge region. The intrinsic flatness ensures that regions can be foliated into parallel line segments orthogonal to the hinge in a unique way, and as with the mean curvature, the final characteristic ensures that there are no jumps in the curvature approximations due to small changes in the triangulation.

To give a large enough sampling of hinges, a union of the vertex regions $V_v$ intersecting the hinge $h$ are used. However, since these unions will not generally be intrinsically flat, boundaries are included within each vertex region that are orthogonal to the hinge $h$.

\begin{definition}[Hinge regions]
	\label{def:V_h}
	The $n$-dimensional region $V_h$ associated with the hinge $h$ is defined as the union of the vertex-dual regions $V_v$, for vertices $v$ in the closure of $h$, intersecting the geodesic extensions $\gamma_h$ of the lines orthogonal to $h$ in $S^n$,
	\begin{equation}
		V_h
		:= \left(\cup_v \, V_v\right)
		\cap \int_h \gamma_h \, \mathrm{d} V^{n-1} ,
	\end{equation}
	for the vertices $v$ that are contained in the closure, or boundary, of the hinge $h$.
\end{definition}

Each region $V_h$ contains all of the hinge $h$, and parts of other hinges, and will overlap with the regions associated with these other hinges. The average curvature orthogonal to each hinge $h$ is given by integrating over these hinge regions.

\begin{figure}[h]
	\centering
	\includegraphics[scale=1]{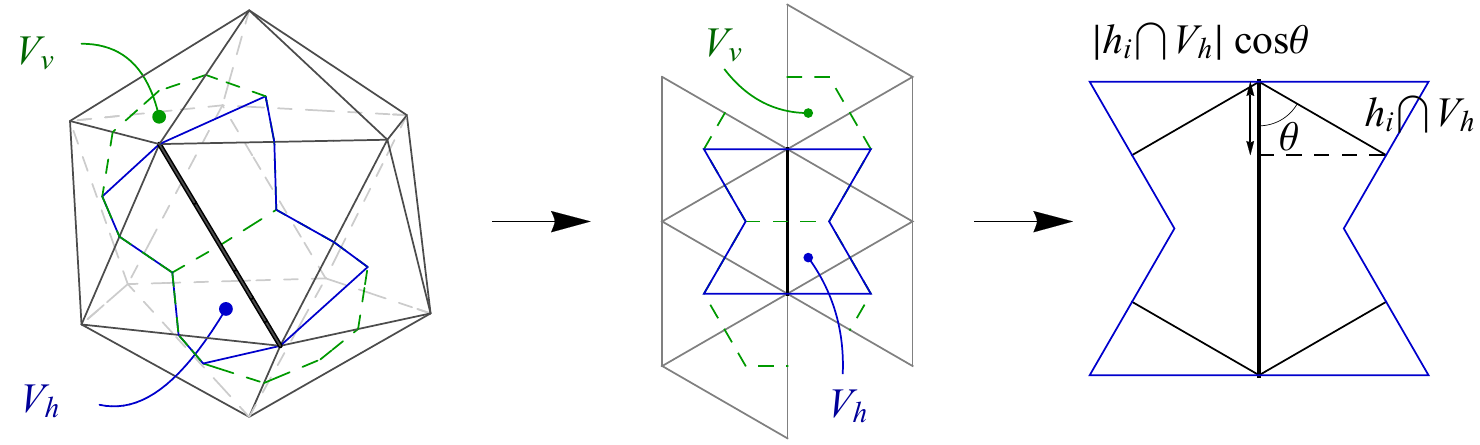}
	\caption{The hinge region for an icosahedron in $\euc{3}$, with the orthogonal projection 
		of $h_i \cap V_h$ onto $h$ shown on the right.}
	\label{fig:kappah}
\end{figure}

\begin{theorem}[Piecewise flat directed curvature]
	\label{thm:kappa_h}
	Assuming the curvature integrals in definitions \ref{def:geoInt} and \ref{def:DirectedInt}, the average curvature over the region $V_h$, in the direction $\hv$ orthogonal to $h$ within $V_h$, 
	\begin{equation}
		\label{kappa_h}
		\kappa_h
		:= \frac{1}{|V_h|} \, \ik_{V_h} (\hv)
		= \frac{1}{|V_h|} 
		\sum_{i} |h_i \cap V_h| \, \cos^2 \theta_i \, \phi_i
		+ O(\phi_i^3)
		,
	\end{equation}
	with the sum taken over all of the hinges $h_i$ that intersect $V_h$, where $\phi_i$ represents the hinge angle at $h_i$, and $\theta_i$ the angle between $h_i$ and $h$ within $S^n$.
\end{theorem}

\begin{proof}
	Since the region $V_h$ is intrinsically flat within $S^n$, it can be foliated into line-segments $\gamma$ orthogonal to the hinge $h$. By definition \ref{def:DirectedInt} and theorem \ref{thm:geoIntS},
	\begin{equation}
		\ik_{V_h} (\hv)
		= \int_{V_h/\gamma} \ik_\gamma d V^{n - 1}
		= \int_{V_h/\gamma} \left[
		\sum_{i} \cos \theta_i \, \phi_i
		+ O(\phi_i^3)
		\right] d V^{n - 1}
	\end{equation}
	with the sum taken over all hinges $h_i$ that each line-segment $\gamma$ intersects. Swapping the sum and integral, so that each hinge $h_i$ can be treated separately,
	\begin{equation}
		\ik_{V_h} (\hv)
		= \sum_{i} \int_{V_h/\gamma} 
		  \left[
		    \cos \theta_i \, \phi_i + O(\phi_i^3) 
		  \right]
		  d V^{n - 1}
		= \sum_{i}
		  \left(
		  \cos \theta_i \, \phi_i + O(\phi_i^3) 
		  \right) \int_{V_h/\gamma} d V^{n - 1},
		  \label{eq:kh-proof1}
	\end{equation}
	with $\left(\cos \theta_i \, \phi_i + O(\phi_i^3)\right)$ factored outside the integral since both $\theta_i$ and $\phi_i$ are invariant along the hinge $h_i$. The orthogonal space $V_h/\gamma$ is the space of the hinge $h$ itself, and for each hinge $h_i$, the integrals above are taken over the portion of $h$ that is intersected by line segments $\gamma$ that also intersect $h_i$ within $V_h$. This is equivalent to the orthogonal projection of $h_i \cap V_h$ onto $h$, as shown on the right of figure \ref{fig:kappah}, so
	\begin{equation}
		\int_{V_h/\gamma} d V^{n - 1} 
		= |h_i \cap V_h| \, \cos \theta_i.
	\end{equation}
	Substituting this into (\ref{eq:kh-proof1}) above,
	\begin{equation}
		\ik_{V_h} (\hv) 
		= \sum_{i} |h_i \cap V_h| \, \cos^2 \theta_i \, \phi_i
		+ O(\phi_i^3),
	\end{equation}
	which can be divided by the $n$-volume of the region $V_h$ to give an average value for the curvature orthogonal to $h$ in $V_h$.
\end{proof}

As with the mean curvature, the piecewise flat directed curvature $\kappa_h$ will change continuously for variations in the triangulation $S^n$ of a smooth manifold $M^n$, as long as the triangulation remains well-defined for a barycentric dual, and remains Delaunay for a Voronoi dual. In particular, the $\cos^2 \theta_i$ part of each term ensures that the contribution from a hinge $h_i$ goes to zero as the hinge gets closer to the boundaries on the interior of a vertex region. 
The $\cos^2 \theta$ coefficient also transforms in the exact way expected of a quadratic form for changes in the angle $\theta$. In fact, for a vertex region $V_v$ that can be cut so that it is intrinsically flat, a sum of expressions of the same form as equation (\ref{kappa_h}) for a set of orthogonal directions can be used to give the mean curvature expression $H_v$ from theorem \ref{thm:H_v}.

In two dimensions, the piecewise flat directed curvature at the hinges (edges) completely determines the curvature within each triangle. For a triangle with edges $\ell_1$, $\ell_2$ and $\ell_3$,
\begin{equation}
	|\ell_1| \, \hv_1
	+ |\ell_2| \, \hv_2
	+ |\ell_3| \, \hv_3
	= 0 ,
	\label{eq:Minkowski}
\end{equation}
for vectors $\hv_i$ orthogonal to $\ell_i$ and tangent to the given triangle. Taking $\hv_1$ and $\hv_2$ as basis vectors, the relation above can be used to give the second fundamental form for mixed arguments. Using the bilinear nature of the second fundamental from, 
\begin{eqnarray}
	\alpha (\hv_1, \hv_2)
	&=& \frac{\alpha (|\ell_1| \, \hv_1, |\ell_2| \, \hv_2)}{|\ell_1| \, |\ell_2|} 
	\nonumber \\
	&=& \frac{1}{2 \, |\ell_1| \, |\ell_2|}\big(
	- \kappa(|\ell_1| \, \hv_1)
	- \kappa(|\ell_2| \, \hv_2)
	+ \kappa(|\ell_1| \, \hv_1 + |\ell_2| \, \hv_2)
	\big) \nonumber \\
	&=& \frac{1}{2 \, |\ell_1| \, |\ell_2|}\big(
	- |\ell_1|^2 \, \kappa_{\ell_1}
	- |\ell_2|^2 \, \kappa_{\ell_2}
	+ |\ell_3|^2 \, \kappa_{\ell_3}
	\big) ,
\end{eqnarray}
with equation (\ref{eq:MixedForm}) used in the second line, and (\ref{eq:Minkowski}) for the final expression. Using this, a complete curvature tensor can be formed for the interior of each triangle.

\section{Computations}
\label{sec:Comp}

\subsection{Surfaces in Euclidean space}
\label{sec:CompEuc}

To compare with existing approaches, computations have been carried out for triangulations of a modified sphere and a peanut-shaped surface in $\euc{3}$, the latter having both positive and negative curvatures. The surfaces are defined in spherical-polar coordinates by the radial functions
\begin{align}
	r_S(\theta, \phi)
	&= \sqrt{
		\left(1 + \frac{1}{4} \sin^2 \theta\right)
		\left(1 + \frac{1}{4} \sin^2 \theta \cos^2 \phi\right)	
	} \ ,
	\\
	r_P(\theta, \phi)
	&= \sqrt{
		\left(\frac{1}{2} + \cos^2 \theta\right)
		\left(1 + \frac{1}{4} \sin^2 \theta \cos^2 \phi\right)	
	} \ .
\end{align}
In order to indicate any convergence to smooth curvature values, triangulations of both surfaces were created using progressively more layers of triangles between the two poles. The resulting numbers of vertices, edges and triangles, and the mean and largest absolute values of the hinge angles are given in table \ref{tab:CompTriang}.

\begin{table}[h!]
	\centering
	\begin{tabular}{lrccc|cc|cc
		}
		&&\multicolumn{3}{c|}{No. simplices}
		&\multicolumn{2}{c|}{Modified sphere}
		&\multicolumn{2}{c}{Peanut-shape}
		\\
		\cline{3-9}
		&
		&Triangles & Edges & Vertices
		&\hspace{0.3cm} Mean \hspace{0.3cm} & Largest
		&\hspace{0.3cm} Mean \hspace{0.3cm} & Largest
		\\
		\cline{3-9}
		\multicolumn{1}{l}{} 
		& 6 Layers 
		& 96 & 144 & 50
		& $18^\circ$ & $32^\circ$
		& $20^\circ$ & $37^\circ$
		\\
		\multicolumn{1}{l}{} 
		& 10 Layers
		& 252 & 378 & 128
		& $11^\circ$ & $18^\circ$
		& $13^\circ$ & $24^\circ$
		\\
		\multicolumn{1}{l}{} 
		& 14 Layers
		& 480 & 720 & 242
		& $8.2^\circ$ & $13^\circ$
		& $9.7^\circ$ & $18^\circ$
		\\
		\multicolumn{1}{l}{} 
		& 18 Layers
		& 780 & 1170 & 392
		& $6.4^\circ$ & $9.8^\circ$
		& $7.6^\circ$ & $14^\circ$
		\\
		\multicolumn{1}{l}{} 
		& 22 Layers
		& 1152 & 1728 & 578
		& $5.2^\circ$ & $8.1^\circ$
		& $6.3^\circ$ & $12^\circ$
		\\
	\end{tabular}
	\caption{The number of triangles, edges and vertices in each simplicial graph, and the mean and largest absolute value of the hinge angles, in degrees, for each triangulation.
	}
	\label{tab:CompTriang}
\end{table}

Along with the new curvature constructions, the cotan formula and Cohen-Steiner and Morvans anisotropic curvature have been computed for each triangulation. Voronoi, or circumcentric dual regions were used to compute both the new piecewise flat mean curvature and the cotan formula. The new piecewise flat directed curvature was computed at each edge, and the anisotropic curvature measure from Cohen-Steiner and Morvan was computed over each Voronoi dual reagion, and divided by the region area to give an average tensor at each vertex. Visual representations of the triangulations and their dual regions are shown in figure \ref{graph:MeanContour}, with a colour-scale used to display the piecewise flat mean curvature of each region.

\begin{figure}[h]
	\centering
	\includegraphics[scale=1]{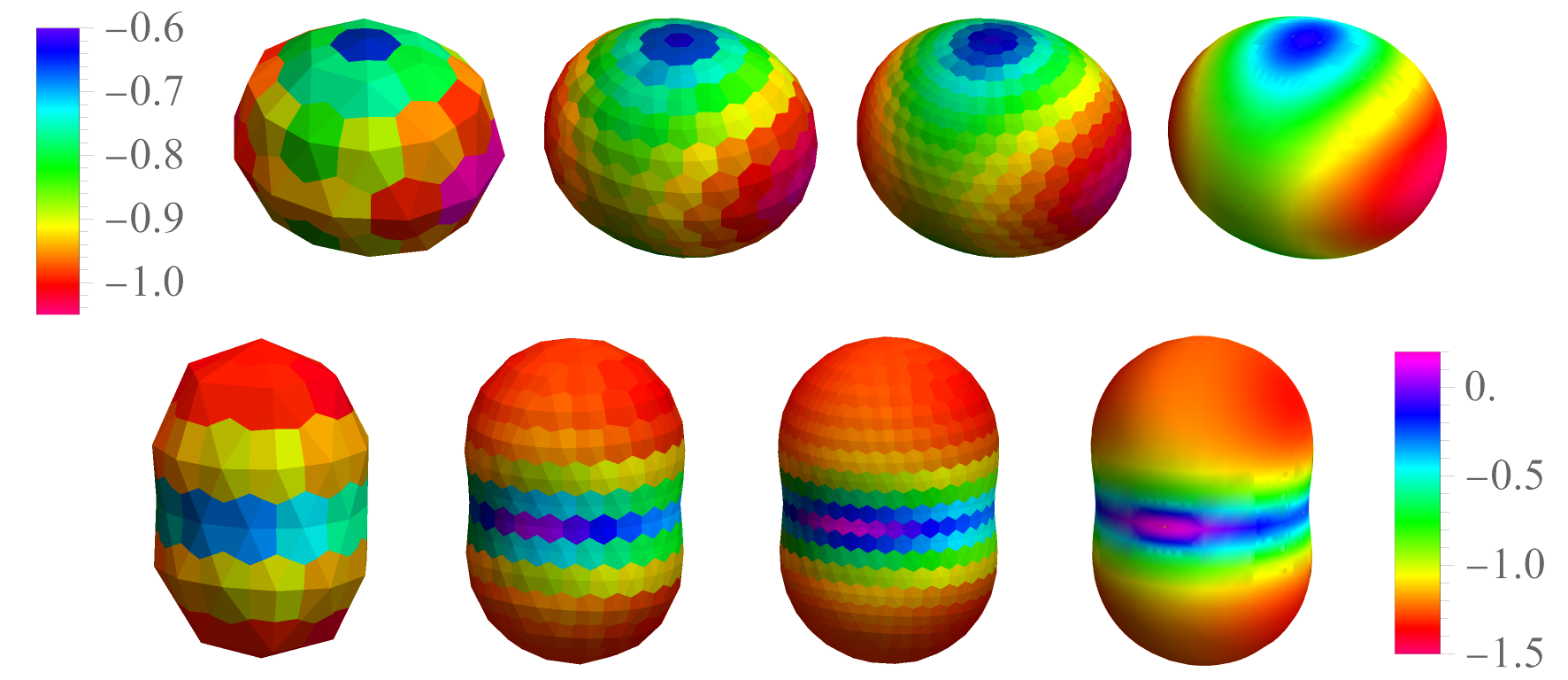}
	\caption{Contour plots showing the piecewise flat mean curvature over each dual region for the 6, 14 and 22 layer triangulations of the modified sphere on top and peanut-shaped surface on the bottom. The smooth mean curvature for each is shown on the far right.}
	\label{graph:MeanContour}
\end{figure}

To measure the effectiveness of each approximation type, the differences between the approximations and their corresponding smooth curvature values are computed. For the new piecewise flat mean curvature and the cotan formula, the value at a given vertex is compared with the smooth curvature at the surface point corresponding with the vertex. For the new piecewise flat directed curvature, the value at an edge is compared with the smooth curvature at the midpoint of the corresponding geodesic segment, in a direction tangent to the surface and orthogonal to the geodesic segment at that point. For the Cohen-Steiner and Morvan curvature, rather than comparing tensors directly, the principal curvatures were computed for the average tensor at each vertex, and compared with the smooth principal curvatures at the corresponding smooth points. In figure \ref{graph:CompErr}, the mean absolute values of these differences are graphed against the mean absolute value of the hinge angles for each triangulation, with error bars denoting the standard deviations of each. To relate the errors for different curvature types, and in a scale-invariant way, numeric values for the errors are presented in table \ref{tab:CompErr} as percentages of the mean absolute value of the principal curvatures for each surface.

\begin{figure}[h]
\centering
\includegraphics[scale=1]{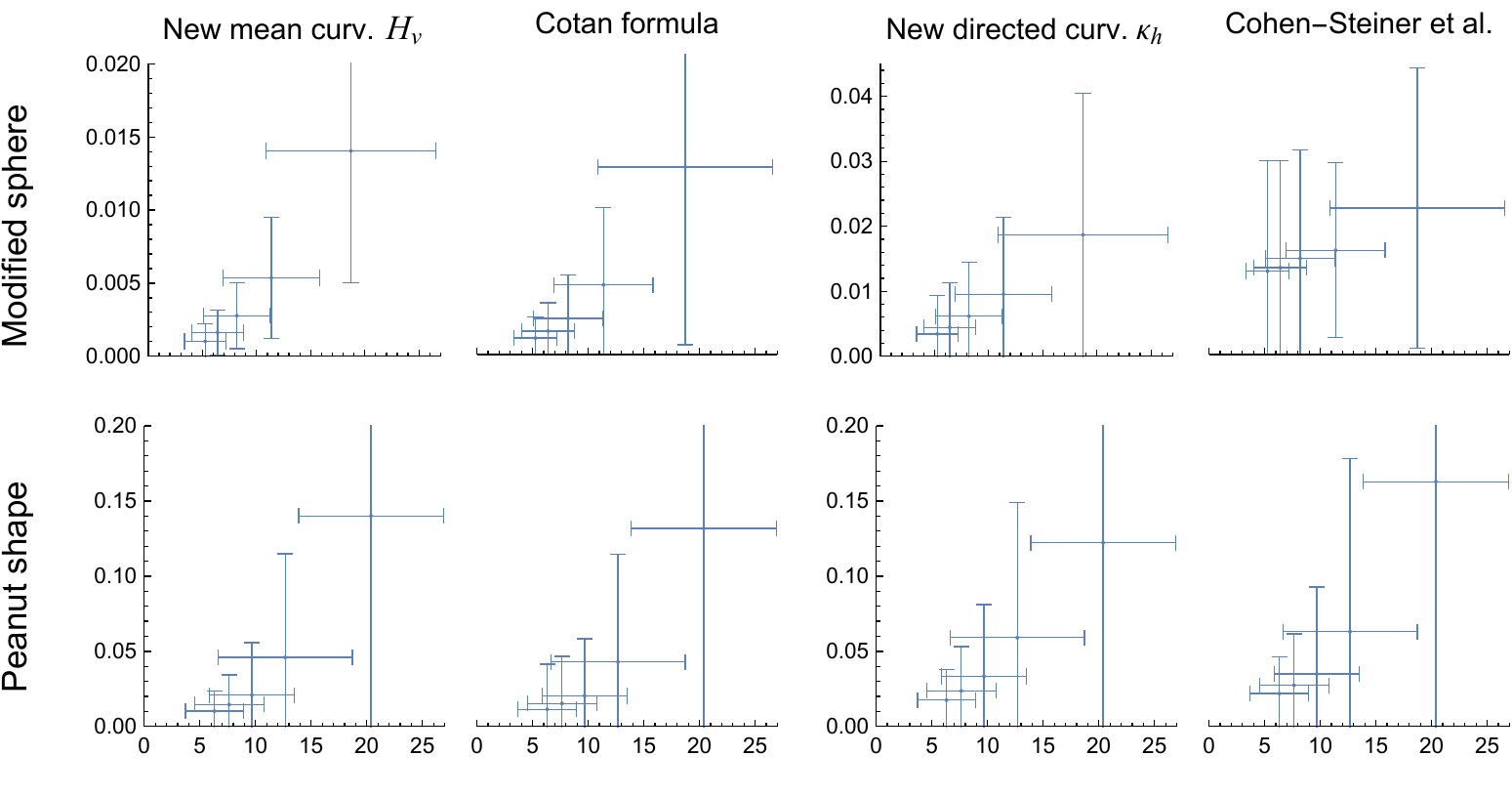}
\caption{The mean absolute errors are plotted against the mean absolute hinge angles, in degrees, with the bars representing the standard deviations. The errors for the new approaches indicate a clear convergence to zero, and have similar values to the other approaches.
}
\label{graph:CompErr}
\end{figure}

\begin{table}[h!]
	\centering
	\begin{tabular}{lrcccc|cccc
		}
		&&\multicolumn{4}{c|}{Modified sphere}
		&\multicolumn{4}{c}{Peanut-shape}
		\\
		\cline{3-10}
		&
		& \ \ Mean \ \ &Cotan
		& Directed &CSM
		& \ \ Mean \ \ &Cotan
		& Directed &CSM
		\\
		\cline{3-10}
		\multicolumn{1}{l}{} 
		& 6 Layers 
		& $1.6\%$ 
		& $1.4\%$ 
		& $2.2\%$ 
		& $2.5\%$ 
		& $12\%$  
		& $12\%$  
		& $11\%$  
		& $14\%$  
		\\
		\multicolumn{1}{l}{} 
		& 10 Layers
		& $0.62\%$ 
		& $0.53\%$ 
		& $1.1\%$  
		& $1.8\%$  
		& $4.1\%$  
		& $3.8\%$  
		& $5.2\%$  
		& $5.6\%$  
		\\
		\multicolumn{1}{l}{} 
		& 14 Layers
		& $0.32\%$ 
		& $0.28\%$ 
		& $0.71\%$ 
		& $1.7\%$  
		& $1.9\%$  
		& $1.8\%$  
		& $3.0\%$  
		& $3.1\%$  
		\\
		\multicolumn{1}{l}{} 
		& 18 Layers
		& $0.19\%$ 
		& $0.18\%$ 
		& $0.51\%$ 
		& $1.5\%$  
		& $1.3\%$  
		& $1.4\%$  
		& $2.1\%$  
		& $2.4\%$  
		\\
		\multicolumn{1}{l}{} 
		& 22 Layers
		& $0.12\%$ 
		& $0.13\%$ 
		& $0.40\%$ 
		& $1.4\%$  
		& $0.9\%$  
		& $1.0\%$  
		& $1.6\%$  
		& $1.9\%$  
		\\
	\end{tabular}
	\caption{The mean errors, as percentages of the mean absolute value of the principal curvatures for each surface. As with the graphs, the values above indicate a clear convergence to zero for the new approaches, and give comparable values to the other approaches.
	}
	\label{tab:CompErr}
\end{table}

The results of the contour plots, error graphs and the table of percentage errors strongly support the effectiveness of the new curvature constructions. All show reasonable approximations for the lowest resolutions, especially when the new mean curvature is viewed as a spatial average in the contour plots. 
The error graphs and table also show a clear convergence to zero as the resolution is increased, with the contour plots showing the emergence of more and more of the finer details visible in the smooth mean curvature plots. 
Finally, the new constructions compare very favourably with the other approaches presented. The errors for the piecewise flat mean curvature are extremely close to those of the cotan formula for all triangulations, and the errors for the piecewise flat directed curvature are on a similar scale to those of the principal curvatures from Cohen-Steiner and Morvan's approach.

\subsection{Non-Euclidean embedding}
\label{sec:CompNonEuc}

For a non-Euclidean embedding space, the spatial part of a Gowdy space-time model \cite{Gowdy} is used, with metric
\begin{equation}
	d s^2
	= e^{0.1 \sin (z)} d x^2
	+ e^{-0.1 \sin(z)} d y^2
	+ d z^2.
\end{equation}
This space is homogeneous in the $x$ and $y$-directions, periodic in the $z$-direction, and can be viewed as a plane gravitational wave, stretching and shrinking alternately in the $x$ and $y$-directions as the $z$-value is changed. The embedded surface is defined by the tangent vectors
\begin{equation}
	\left<0, \, 1, \, 0\right>, \ 
	\frac{1}{3} \left<-1, \, 0, \, \pi\right>,
	\label{eq:GowdySurf}
\end{equation}
giving a coordinate plane tilted away from the $z$-axis, which was deemed to have a simple-enough form but with an interesting variation in extrinsic curvature. The positive orientation of the surface is defined as the side with the positive $x$-direction.

\begin{figure}[h]
	\centering
	\includegraphics[scale=1]{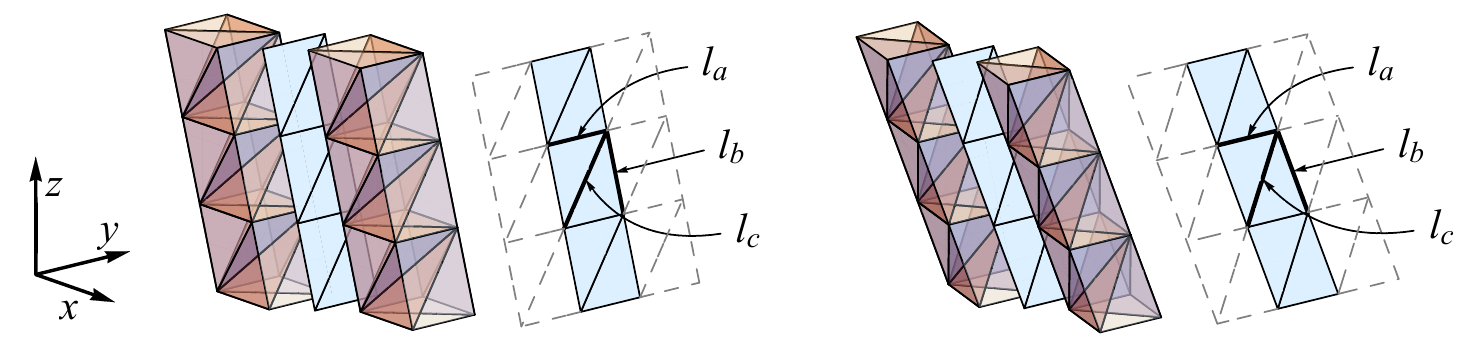}
	\caption{The rectangular grid is shown on the left and skewed grid on the right, along with the three-dimensional triangulations that are used to determine the hinge angles.}
	\label{fig:GowdyTriang}
\end{figure}

The surface is triangulated using both a rectangular grid, aligned with the vectors in (\ref{eq:GowdySurf}), and a grid that is skewed to ensure the triangulation remains Delaunay, as shown in figure \ref{fig:GowdyTriang}. The vectors tangent to the graph segments on the surface are given below,
\begin{align}
	\textrm{Rectangular: } \ \
	& \vec{v}^R_a = \left<0, \, 1, \, 0\right>, \
	&&\vec{v}^R_b = \frac{1}{3} \left<-1, \, 0, \, \pi\right>, \
	&&\vec{v}^R_c = \vec{v}^R_a + \vec{v}^R_b,
	\nonumber \\
	\textrm{Skewed: } \ \
	& \vec{v}^S_a = \left<0, \, 1, \, 0\right>, \
	&&\vec{v}^S_b = \frac{1}{3} \left<-1, \, -\frac{2}{3}, \, \pi\right>, \
	&&\vec{v}^S_c = \vec{v}^S_a + \vec{v}^S_b.
	\label{eq:GowdyTriVecs}
\end{align}
Since changes in both the surface and embedding space only occur as the $z$-value is changed, grids with $6$, $12$, $24$ and $48$ blocks over a full period in the $z$ direction are used to give different resolution triangulations, with the grid widths re-scaled accordingly. As explained in section \ref{sec:PFnonE}, the hinge angles are defined by constructing a three-dimensional triangulation around the surface, as shown in figure \ref{fig:GowdyTriang}. Details of these triangulations can be found in \cite{PLCurv} and \cite{PLRFGowdy}, where they were used to give piecewise flat approximations of the scalar and Ricci curvature in the first, and the Ricci flow of the Gowdy three-space in the latter.

The mean curvature was computed over barycentric duals for the rectangular grid, since the triangulations are not Delaunay, and Voronoi or circumcentric duals for the skewed grid. As the surface and embedding space are homogeneous in the $x$ and $y$ directions, the curvatures can be plotted as functions of the $z$ coordinate, with graphs of the mean curvature shown in figure \ref{graph:GowdyMean} and the directed curvature for each edge-type in figure \ref{graph:GowdyExt}. The approximations are graphed as piecewise linear curves, with the curvature approximations represented by the points where segments join. For the mean curvature, the approximations are located at the $z$-value of the smooth point corresponding with the piecewise flat vertex, and for the directed curvature, the $z$-values represent the midpoint of the geodesic segment corresponding with the given edge. The smooth mean curvature is also graphed as a function of $z$ in figure \ref{graph:GowdyMean}, and for each edge-type in figure \ref{graph:GowdyExt}, the smooth curvature at the midpoint of the corresponding geodesic segments, tangent to the surface and orthogonal to the corresponding vectors in (\ref{eq:GowdyTriVecs}), are also graphed.

\begin{figure}[h]
	\centering
	\includegraphics[scale=1]{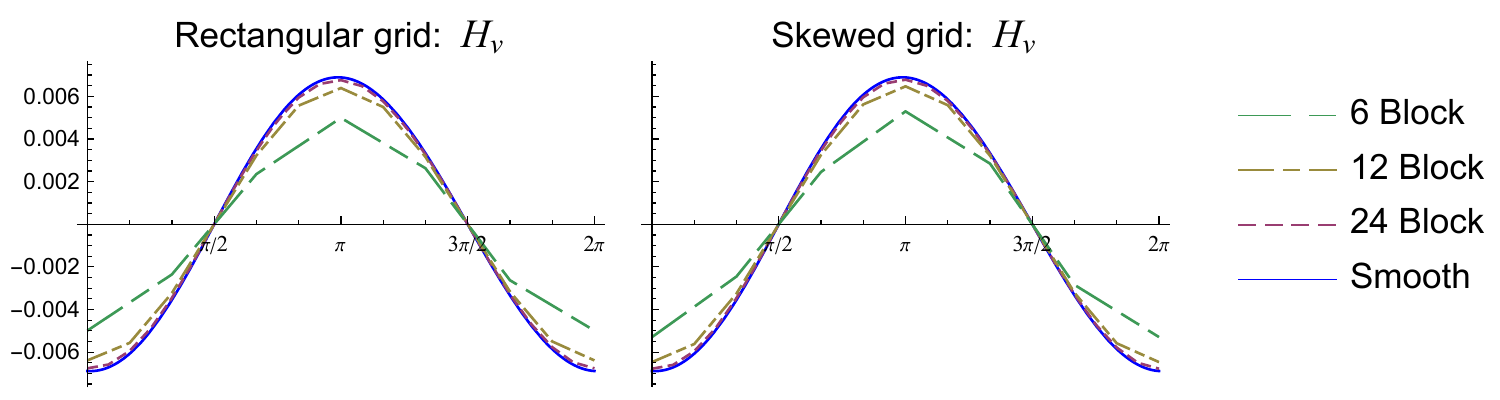}
	\caption{Piecewise linear curves are used to represent the piecewise flat mean curvature approximations as functions of the $z$-coordinate, showing a convergence to the smooth mean curvature as the resolution is increased.}
	\label{graph:GowdyMean}
\end{figure}

\begin{figure}[h]
	\centering
	\includegraphics[scale=1]{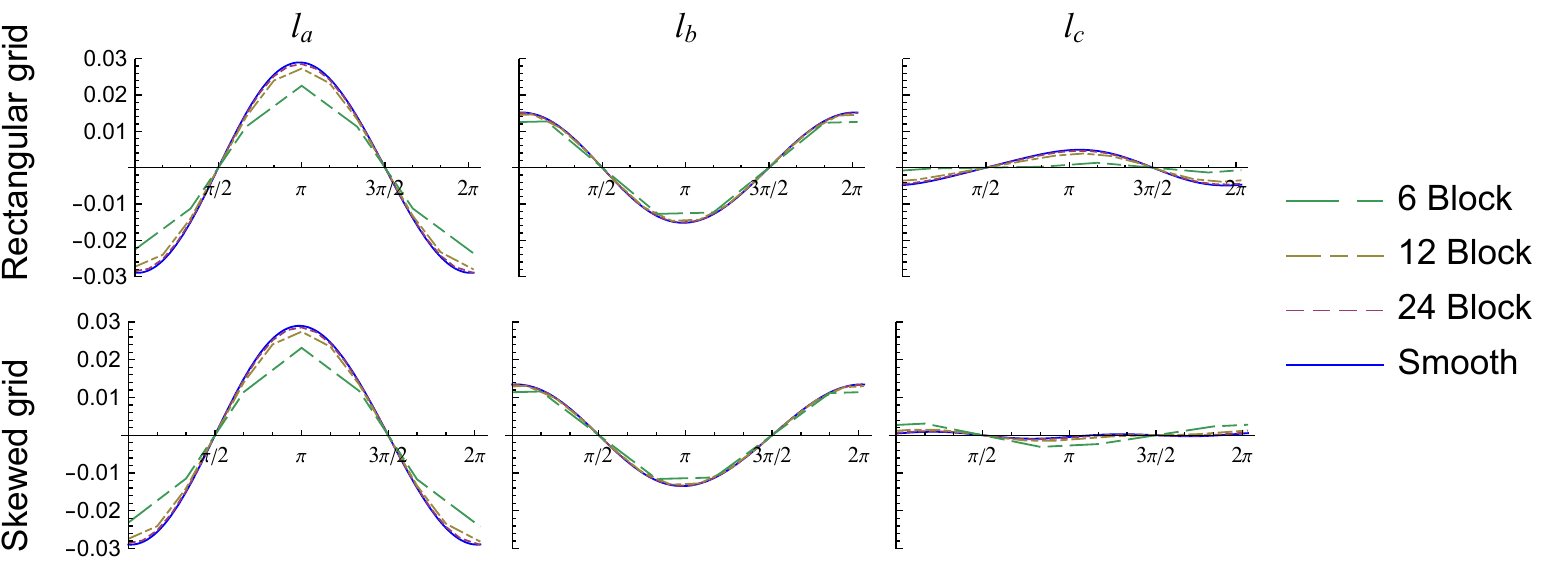}
	\caption{The piecwise flat directed curvatures orthogonal to each edge-type, and tangent to the surface, are shown as piecewise linear functions of the $z$-coordinate, clearly approaching the corresponding smooth directed curvatures as the resolution is increased.}
	\label{graph:GowdyExt}
\end{figure}

As in the previous section, numeric values for the errors have also been computed. The mean curvature approximation at each vertex was compared with the smooth mean curvature at the corresponding $z$-value, and the piecewise flat directed curvature at each edge was compared with the smooth curvature at the midpoint of the corresponding geodesic segment, in the direction tangent to the surface and orthogonal to the geodesic segment at the midpoint. In order to compare the different curvature types in a scale-invariant way, table \ref{tab:Gowdy} presents the mean of the absolute values of the errors, as percentages of the mean absolute value of the smooth principal curvatures. To compare the effectiveness of each triangulation, the mean of the absolute values of the hinge angles, in degrees, are also given.

\begin{table}[h!]
	\centering
	\begin{tabular}{lrccc|ccc}
		&&\multicolumn{3}{c|}{Rectangular Grid}
		&\multicolumn{3}{c}{Skew Grid}
		\\
		\cline{3-8}
		&
		&Hinge Ang. & \ \ Mean \ \ 
		& \ Directed \
		&Hinge Ang. & \ \ Mean \ \ 
		& \ Directed \
		\\
		\cline{3-8}
		\multicolumn{1}{l}{} 
		& 6 Blocks 
		& $0.48^\circ$ 	
		& $9.0\%$  		
		& $16\%$  		
		& $0.52^\circ$ 	
		& $7.6\%$  		
		& $13\%$  		
		\\
		\multicolumn{1}{l}{} 
		& 12 Blocks
		& $0.27^\circ$ 	
		& $2.2\%$ 		
		& $4.4\%$ 		
		& $0.29^\circ$ 	
		& $1.8\%$ 		
		& $3.5\%$ 		
		\\
		\multicolumn{1}{l}{} 
		& 24 Blocks
		& $0.14^\circ$ 	
		& $0.57\%$ 		
		& $1.1\%$ 		
		& $0.15^\circ$ 	
		& $0.47\%$ 		
		& $0.9\%$ 		
		\\
		\multicolumn{1}{l}{} 
		& 48 Blocks
		& $0.07^\circ$ 	
		& $0.14\%$ 		
		& $0.28\%$ 		
		& $0.08^\circ$ 	
		& $0.12\%$ 		
		& $0.22\%$ 		
		\\
	\end{tabular}
	\caption{
		The mean absolute values of the hinge angles, and the mean errors, as percentages of the mean absolute value of the smooth principal curvatures for the surface.
	}
	\label{tab:Gowdy}
\end{table}

The results strongly support the effectiveness of the new constructions, and their use for surfaces embedded in non-Euclidean spaces. For the two lower resolutions, involving only 12 and 24 triangles, the graphs closely model the behaviour of the smooth curvature, with reasonable percentage errors. The graphs then show a clear convergence to the smooth curvature, and the percentage errors decrease by a factor of about a quarter each time the number of triangles is doubled, a square of the approximate factor that the hinge angles reduce by. The two triangulation types also agree closely with each other, with exceptionally similar graphs for the mean curvature, and for the directed curvature at the $\ell_a$ edges, which are common to both triangulation types. Interestingly, while the hinge angles are lower for the rectangular grid, the curvature approximations have lower mean errors for the skewed grid, on the order of $10\%$ lower in most cases. This is likely due to the triangles being more strongly Delaunay, and the resulting use of Voronoi duals, properties that seem to improve other approaches as well. However, it is reassuring to see that the errors do not increase dramatically for non-Delaunay triangulations.

To demonstrate the effectiveness of the higher resolution approximations, the directed curvature at the diagonal $\ell_c$ edges are graphed for the 24 and 48 block triangulations in figure \ref{graph:GowdyExtC}, with the 48 block approximations almost indistinguishable from the smooth curvature. These are particularly impressive approximations, with scales that are an order of magnitude lower than the curvature values for the $\ell_a$ edges, and a more complicated bahaviour. The computations for these curvatures are also dominated by the hinge angles at the $\ell_a$ and $\ell_b$ edges, making this behaviour impossible to achieve with a single hinge approach. For the rectangular grid, the hinge angles at $\ell_c$ are multiple orders of magnitude less than at $\ell_a$ (4 orders less for the 48-block triangulation), and for the skewed grid, the hinge angles at $\ell_c$ approximate a cosine curve with a period of $2 \pi$, quite different to the behaviour displayed on the right of figure \ref{graph:GowdyExtC}.

\begin{figure}[h]
	\centering
	\includegraphics[scale=1]{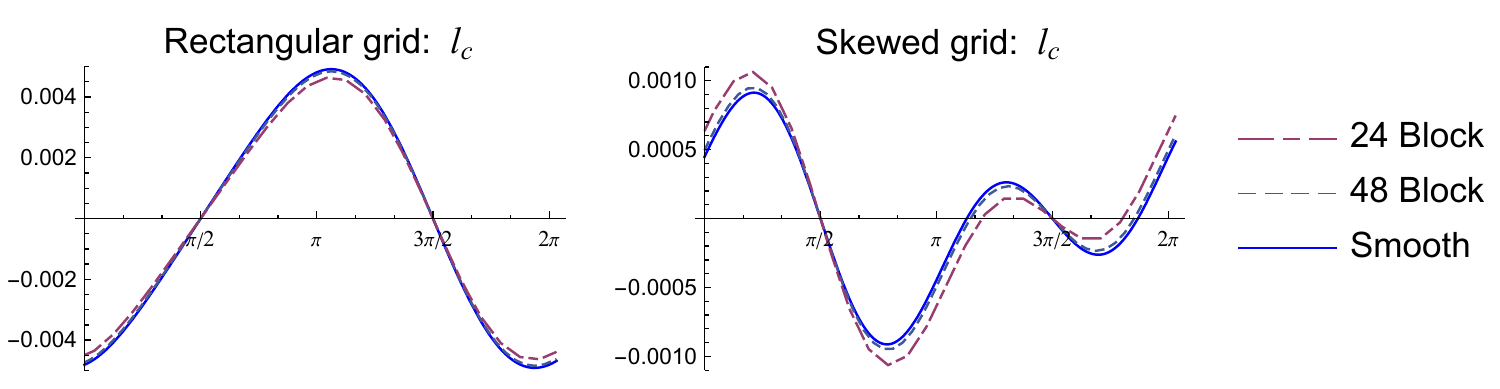}
	\caption{
		Piecewise flat directed curvature at the diagonal $\ell_c$ edges, similar to the graphs on the right of figure \ref{graph:GowdyExt}, but re-scaled. The higher resolutions give remarkably close approximations, particularly considering the more complicated behaviour and smaller scale.
	}
	\label{graph:GowdyExtC}
\end{figure}

\section{Conclusion}
\label{sec:Con}

Discrete forms of the mean and directed curvature have been constructed on piecewise flat approximations of smooth manifolds in a manner that aims to approximate the local smooth curvature and be applicable to manifolds embedded in non-Euclidean spaces. 
When dual regions $V_v$ are chosen to be Voronoi or barycentric, the resulting mean and directed curvature expressions, $H_v$ and $\kappa_h$ respectively, take the slightly simplified form 
\begin{align}
	H_v \
	= \ \frac{1}{n \, |V_v|}
	&\sum_{i} \, \frac{1}{2}|h_i| \, \phi_i ,
	\\
	\kappa_h \
	= \ \ \ \frac{1}{|V_h|} 
	&\sum_{i} \, \frac{1}{2}|h_i| \, \phi_i \, \cos^2 \theta_i
	,
\end{align}
where $h_i$ are the hinges that intersect the regions $V_v$ or $V_h$, $\phi_i$ are the corresponding hinge angles, and $\theta_i$ are the angles between the hinges $h_i$ and $h$. The hinge region $V_h$ is formed by the union of dual regions $V_v$, for vertices $v$ in the closure of the hinge $h$, bounded on their interior by regions orthogonal to $h$ within $S^n$, see figure \ref{fig:kappah} for example. The piecewise flat directed curvature $\kappa_h$ approximates the smooth curvature in a direction orthogonal to the hinge $h$ and tangent to the manifold.

Computations for a pair of surfaces in Euclidean space, and two different triangulation types for a surface in a non-Euclidean space, show these new curvatures to:
\begin{itemize}
	\item reasonably approximate the smooth curvature for low resolution triangulations and converge to the corresponding smooth curvature values as the resolution is increased;
	
	\item give consistent values for different triangulations of the same surface;
	
	\item apply to surfaces in both Euclidean and non-Euclidean spaces;
	
	\item closely match the results of other approaches for Euclidean embeddings.
\end{itemize}
The convergence of the new mean curvature is clearly demonstrated by the contour plots in figure \ref{graph:MeanContour}, and the convergence of the new directed curvature for a non-Euclidean embedding is remarkably clear for the two different triangulation types in figure \ref{graph:GowdyExtC} above.

For a piecewise flat approximation of a two-dimensional surface, a complete curvature tensor can be formed within each triangle, using only the directed curvature at the edges. Though there are not enough hinges on the boundary of an $n$-simplex in higher dimensions, there are the ideal number and orientation of edges. Unfortunately, 
constructing an edge-tangent curvature on a region similar to the hinge volume has not been successful. This is mostly due to the $\cos^2 \theta_i$ term having the opposite effect here, with hinges on the boundary having the highest weighting, leading to potential jumps in the curvature for slight variations in the triangulation. Work is currently underway to test different types of edge-regions.

The discrete \emph{intrinsic} curvatures in \cite{PLCurv} were developed following a similar approach to this paper, and the structures of the two fit together very closely, particularly when the piecewise flat manifold $S^n$ acts as the boundary of an $(n+1)$-dimensional piecewise flat manifold $R^{n+1}$. In this situation, the hinges of $S^n$ coincide with the co-dimension-two simplices of $R^{n+1}$, where the intrinsic deficit angles are defined, a connection that has already been used by Hartle and Sorkin \cite{HS81} to provide boundary terms for the Regge Calculus action using the hinge angles. On top of this, the vertex and hinge regions $V_v$ and $V_h$ act as boundaries for the regions used to construct the scalar and sectional curvature in \cite{PLCurv}. This makes the new piecewise flat extrinsic curvatures well suited to extend the piecewise flat Ricci flow of \cite{PLCurv} and \cite{PLRFGowdy} to manifolds with boundary.

\

\noindent
{\bf Data Availability:} \ The code used to perform the computations in section \ref{sec:Comp}, and the data generated, are available in the Zenodo repository, https://doi.org/10.5281/zenodo.7787062.

\section*{Acknowledgment}

This paper is dedicated to the memory of Niall {\'O} Murchadha. Niall was a patient teacher, a supportive mentor and a good friend. I will always remember your insightful perspectives and contagious enthusiasm. You are sadly missed.

\bibliography{Ref}
\bibliographystyle{unsrt}

\end{document}